\documentclass[twocolumn]{autart}    

\usepackage{graphicx}          

\usepackage{amssymb}
\usepackage{amsmath}
\usepackage{xcolor}

\newtheorem{proposition}{Proposition}
\newtheorem{theorem}{Theorem}
\newtheorem{corollary}{Corollary}
\newtheorem{assumption}{Assumption}
\newtheorem{definition}{Definition}

\usepackage[authoryear,longnamesfirst]{natbib}

\begin{document}
\begin{frontmatter}

\title{A necessary and sufficient condition for discrete-time consensus on star boundaries  \thanksref{footnoteinfo}}

\thanks[footnoteinfo]{This paper was not presented at any IFAC 
meeting. Corresponding author is Johan Thunberg.}

\author[Lund]{Galina Sidorenko}\ead{galina.sidorenko@eit.lth.se},               
\author[Lund]{Johan Thunberg}\ead{johan.thunberg@eit.lth.se}    

\address[Lund]{Department of Electrical and Information Technology, Faculty of Engineering, Lund University}  

\begin{abstract}     
It is intuitive and well known, that if agents in a multi-agent system iteratively update their states in the Euclidean space as convex combinations of neighbors’ states, all states eventually converge to the same value (consensus), provided the interaction graph is sufficiently connected. However, this seems to be also true in practice if the convex combinations of states are mapped or radially projected onto any unit $l_p$-sphere or even boundaries of star-convex sets, herein referred to as star boundaries. In this paper, we present insight into this matter by providing a necessary and sufficient condition for asymptotic consensus of the normalized states (directions) for strongly connected directed graphs, which is equivalent to asymptotic consensus of states when the star boundaries are the same for all agents. Furthermore, we show that when asymptotic consensus occurs, the states converge linearly and the point of convergence is continuous in the initial states. Assuming a directed strongly connected graph provides a more general setting than that considered, for example, in gradient-based consensus protocols, where symmetric graphs are assumed. Illustrative examples and a vast number of numerical simulations showcase the theoretical results. 
\end{abstract}

\end{frontmatter}

\section{Introduction}
A key problem in multi-agent systems is the consensus problem, where agents must agree on a common decision with often limited knowledge of a global coordinate system~\citep{olfati2004consensus,olfati2007consensus,ren2005consensus,ren2005survey,garin2010survey}. 
The nature of this problem depends heavily on the context; whether it involves achieving agreement on numerical values~\citep{schenato2011average,savazzi2020federated}, aligning physical states~\citep{thunberg2016consensus,olfati2006flocking}, or coordinating behaviors~\citep{cao2013overview,ren2005survey,reynolds1987flocks,vicsek1995novel}.

In the paper, we study discrete-time consensus on a large class of boundaries of star-convex sets, or \textit{star boundaries}. These are shapes homeomorphic to the unit $2$-sphere in $d$-dimensional Euclidean space. This, of course, includes the \( l_p \)-spheres as special cases. When $p = 2$, the problem comprises consensus on the $d$-dimensional unit sphere~\citep{dorfler2014synchronization,scardovi2007synchronization,sarlette2009synchronization,thunberg2016consensus,markdahl2017almost}. Consensus for unit vectors (w.r.t. Euclidean norm) is a special case of consensus for orthogonal matrices~\citep{tron2012intrinsic,markdahl2020high,thunberg2016consensus,thunberg2018dynamic}, and more generally Riemannian  manifolds~\citep{tron2012riemannian,shah2017distributed}.

However, some important questions remain unanswered for the unit sphere and these more general surfaces considered. Similar to nonlinear consensus for multi-agent systems evolving in Euclidean space~\citep{moreau2004stability,lin2005state,wu2024hilbert,thunberg2017local}, consensus seemingly (at least almost) always is achieved globally when projected gradient descent-like algorithms are used, even when the interaction graph of the multi-agent system is not symmetric. We take a step toward understanding this phenomenon by introducing a necessary and sufficient condition for consensus, which provides more intuition behind this behavior. We show that a continuous function of the initial state must be equal to zero for consensus not to occur, implying that the set of initial conditions for which consensus is not achieved is a closed set.

We consider interactions between the agents in the system described by a directed strongly connected graph, where the edge-weights of the graph appear as coefficients in the consensus algorithm/protocol. Graph-based methods have been extensively used to describe and analyze distributed multi-agent coordination~\citep{mesbahi2010graph}.  If the graph is connected and symmetric, and the Euclidean unit sphere is considered, the algorithm in this paper reduces to projected gradient ascent (or descent, depending on how we choose to formulate the objective). In this context, it falls under the umbrella of first order methods for optimization on manifolds~\citep{diaconis1991geometric,absil2012projection,boumal2019global,balashov2021gradient,karimi2016linear,scutari2019distributed,yuan2016convergence}. In such contexts, restrictive assumptions are often imposed, such as requiring the adjacency matrix or mixing matrix to be doubly stochastic~\citep{shi2015extra,nedic2017achieving}. 

Recently, many new results on contraction theory and generalizations thereof have been published for dynamic systems that share similarity with the ones considered here. These include introduction and results on weak and semi-contraction theory~\citep{jafarpour2021weak,de2023dual}, contraction on polyhedral cones~\citep{jafarpour2024monotonicity} and coupled oscillators~\citep{jafarpour2021weak}, and consensus~\citep{deplano2023novel}. Our update equation for the states in the consensus algorithm comprises a radial projection of conical combinations of neighboring states onto star boundaries. This operation is not required to be differentiable, and the state-domain is neither convex nor geodesically convex. This prevents us from using classic contraction results and more recent results on contraction analysis of differentially positive discrete-time deterministic systems~\citep{kawano2025contraction}. 
Similarly to many of these works, we analyze a system that evolves in a positive cone. This can be seen as a tool to show our main results for the original system. However, contraction theory is not used to derive the main contribution.

The main contribution of this paper is a necessary and sufficient condition for asymptotic consensus of directions, meaning that the normalized state vectors asymptotically converge to the same direction. Furthermore, we show that whenever asymptotic consensus of directions occurs, the normalized states converge linearly. We provide some immediate assumptions on the initial states under which the  condition is satisfied. The region of initial states for which asymptotic consensus of directions occurs is an open set, and the consensus point is locally continuous in the initial state. Finally, extensive numerical simulations and inspection of the necessary and sufficient condition seem to suggest a conjecture that asymptotic consensus of directions occurs for all initial states but a set of measure zero.

The rest of the paper is organized as follows. After the preliminaries given in Section \ref{sec:preliminaries}, we introduce the algorithm in Section \ref{sec:algorithm}. Section \ref{sec:convergence_matrices} presents theoretical results on the convergence of a class of time-varying matrices, which are leveraged to show the main result in Section \ref{sec:convergence_to_consensus}. The theoretical results are supported by the illustrative examples presented in Section \ref{sec:simulations}, derived from numerous simulations conducted with random system parameters, in which asymptotic consensus of directions was observed in every simulation.
The paper concludes with a brief summary in Section \ref{sec:conclusions}.

\section{Preliminary notation and definitions} \label{sec:preliminaries}
\subsection{Vectors, matrices and graphs}
Throughout, the convention is to use row-vectors instead of column-vectors. Thus, a vector $x$ denotes a row-vector, while $x^T$ represents a column-vector. When we write $x \in \mathbb{R}^{d}$, this is equivalent to $x \in \mathbb{R}^{1\times d}$. Thus, $x^T \in \mathbb{R}^{d \times 1}$ is a column-vector with $d$ elements. We denote the $(n \times m)$-dimensional matrix of ones and the $(n \times m)$-dimensional matrix of zeros as $\textbf{1}_{n,m}$ and $\textbf{0}_{n,m}$, respectively. The identity matrix of size $(n \times n)$ is denoted by $I_n$. For a vector $x$, $[x]_i$ is the $i$-th element of the vector. 
For a matrix $A \in \mathbb{R}^{n \times d}$, we use the notation $[A]_i$ for the $i$-th row and $[A]_{ij}$ for the element in the $i$-th row and $j$-th column. To simplify notation, we sometimes introduce a symbol $a_{ij}$ for each $[A]_{ij}$ and then write $A = [a_{ij}]$.  

A directed graph with nodes $\mathcal{V} = \{1, 2, \ldots, n\}$ has notation $\mathcal{G}(n) = (\mathcal{V}(n), \mathcal{E})$ where $\mathcal{E} \subset (\mathcal{V}(n) \times \mathcal{V}(n))$ is the set of directed edges. 

We say a matrix $A = [a_{ij}] \in \mathbb{R}^{n \times n}$ is a \textit{nonnegative weight matrix} for the directed graph $\mathcal{G}(n) = (\mathcal{V}(n), \mathcal{E})$ if $a_{ij} > 0$ for $(i,j) \in \mathcal{E}$ and $a_{ij} = 0$ for $(i,j) \notin \mathcal{E}$. 

We define $A$ as a \textit{diagonally dominant} nonnegative weight matrix for the graph $\mathcal{G}(n) = (\mathcal{V}(n), \mathcal{E})$ if $A$ is a nonnegative weight matrix which satisfies the condition 
\begin{align}
    [A]_{ii} > \sum_{j \neq i }[A]_{ij}, \forall  i \in \{1,2,\dots, n\}.
\end{align}
A matrix $D  \in \mathbb{R}^{n \times n}$ is called a \textit{positive diagonal} matrix if all elements on the main diagonal are positive real numbers, and all off-diagonal elements are zero. Any positive diagonal matrix $D$ can be defined by a positive vector $(d_1, \dots d_n)$, $d_i>0$. In this case, we write $D=\text{diag}\left (d_1, \dots d_n \right )$. 

Let $B\in \mathbb{R}^{n \times n}$ be a nonnegative matrix. We define the \textit{minimum cone} as the set comprising all conical sums of the rows of $B$:
\begin{align}
\mathcal{C}(B) = \left\{ \sum_{i=1}^{n} \lambda_i [B]_i : \lambda_i \geq 0, \, \forall i \in \{1, 2, \dots, n\} \right\}.
\end{align}
This is the smallest cone that contains all the rows of matrix $B$. This cone is invariant under nonnegative matrix multiplication. Specifically, for any nonnegative matrix $W \in \mathbb{R}^{n \times n}$, it holds that $\mathcal{C}(WB) \subseteq \mathcal{C}(B)$.

For a matrix $A\in \mathbb{R}^{n \times d}$ with no zero-rows, we define the following function: 
\begin{align} \label{eq:phi_max}
\phi_{\max}(A) = \max_{1 \leq i, j \leq n} \arccos \left( \frac{[A]_i [A]_j^\top}{\|[A]_i\|_2 \|[A]_j\|_2} \right),
\end{align}
which is the largest angle between any two normalized rows of the matrix $A$. 

It holds that 
\begin{align} \label{eq:cos_min}
&\cos(\phi_{\max}(A)) = \min_{1 \leq i, j \leq n}   \frac{[A]_i [A]_j^\top}{\|[A]_i\|_2 \|[A]_j\|_2}.
\end{align}

\subsection{Star boundaries and radial projections}~\label{subsection:star}
The unit sphere in $\mathbb{R}^d$ is the set
\begin{align}
\mathbb{S}^{d-1} = \left\{ x \in \mathbb{R}^d : \|x\|_2 = 1 \right\}.
\end{align}
We define a \textit{directional function} on $\mathbb{S}^{d-1}$ as a continuous function $\gamma: \mathbb{S}^{d-1} \rightarrow \mathbb{R}^+$ that is bounded by a positive constant from above and bounded by a positive constant from below. 
For such a directional function $\gamma$, we define 
\begin{align}
  \mathbb{S}^{d-1}_{\gamma} 
  = \{ 
  y \in \mathbb{R}^d: 
    \exists x \in \mathbb{S}^{d-1} \text{ s.t. } y = \gamma(x)x\}.&
\end{align}
The set $\mathbb{S}^{d-1}_{\gamma}$ is the boundary of a star-convex set at the origin or a \textit{star boundary} of a \textit{star set}. A set $\mathcal{S} \subset \mathbb{R}^d$ is a star set at the origin if $x \in \mathcal{S}$ implies $\alpha x \in \mathcal{S}$ for $0 \leq \alpha \leq 1$. 
A sphere with respect to the $l_p$-norm (or simply the $p$-norm) with radius $r$ is a special case of a star boundary,  
where the directional function is $\gamma(x)=~\frac{r}{\|x\|_p}$ for $x\in ~\mathbb{S}^{d-1}$. For the special case of the unit $l_p$-sphere, we introduce the notation $\gamma_{l_p}$ for the directional function. 

Now, for $n$ directional functions represented as $\bar \gamma = (\gamma_1, \gamma_2, \ldots, \gamma_n)$ we define 
\begin{align}
\mathbb{S}(n,d,\bar \gamma) = \{X \in \mathbb{R}^{n \times d} : [X]_i \in \mathbb{S}^{d-1}_{\gamma_i}, ~\forall  i \in \{1,\dots, n\}\}.
\end{align}
The $i$'th row of a matrix in $\mathbb{S}(n,d,\bar \gamma)$ is an element of $\mathbb{S}^{d-1}_{\gamma_i}$. The set of matrices in $\mathbb{R}^{n \times d}$ for which each row has unit length w.r.t the Euclidean norm is $\mathbb{S}(n,d,\text{id})$, where $\text{id}$ is the identity map. We introduce $\mathbb{S}(n,d)$ as shorthand notation for this set. For $p\in\mathcal{N}\cup \infty$, where $\mathcal{N}=\{1,2,\dots\}$, we define $\bar \gamma_{l_p} = (\gamma_{l_p}, \gamma_{l_p}, \ldots, \gamma_{l_p})$. 
With this notation, the set of row-stochastic matrices is $\mathbb{S}(n,d,\bar \gamma_{l_1})$.

For a directional function $\gamma$, we define the \textit{radial projection} from 
$\mathbb{R}^d\backslash\{0\}$ to $ \mathbb{S}^{d-1}_{\gamma}$: 
\begin{equation}\label{eq:projection1}
P_{\mathbb{S}^{d-1}_{\gamma}}(x) = \gamma\left(\frac{x}{\|x\|_2}\right) \frac{x}{\|x\|_2}.
\end{equation}
An illustration of this projection for several sets 
is given in Fig.~\ref{fig:projection}.

\begin{figure}[h!]
    \centering
\includegraphics[width=0.65\columnwidth]{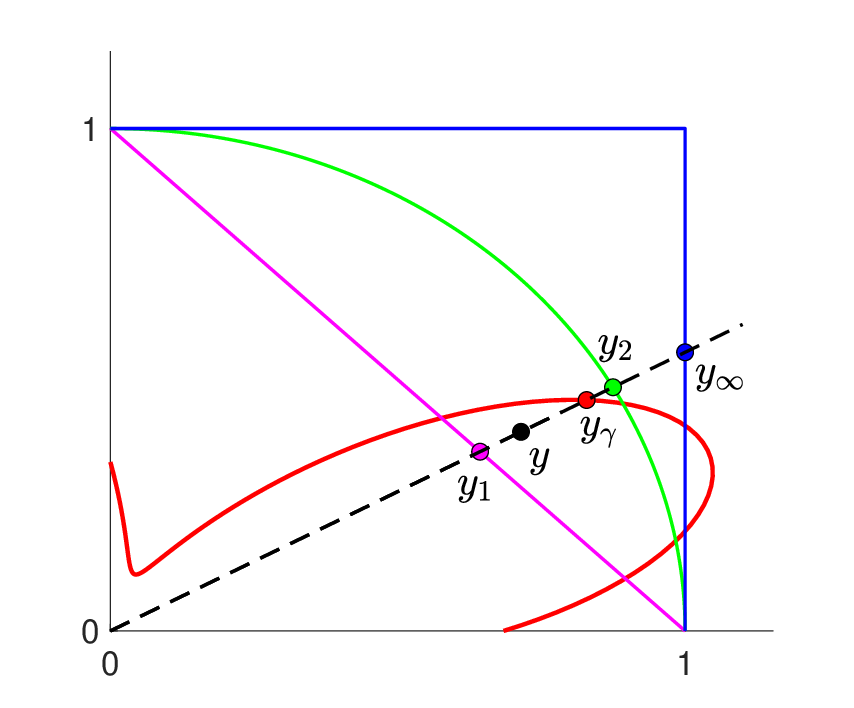} 
    \caption{Projection of a point $y$ onto four different star boundaries according to \eqref{eq:projection1}, where $y_1$, $y_2$, $y_{\infty}$, and $y_{\gamma}$ are the projections onto the unit $l_1$-sphere (magenta), the unit $l_2$-sphere (green), the unit $l_{\infty}$-sphere (blue), and a more complicated star boundary (red), respectively. Here, $d=2$.}
    \label{fig:projection}
\end{figure}

For a matrix $X \in \mathbb{R}^{n \times d} $ where each row is in $\mathbb{R}^d\setminus\!\{0\}$, 
we define the row-wise radial projection $P_{\mathbb{S}(n,d,\bar \gamma)}$ onto $\mathbb{S}(n,d,\bar \gamma)$ as the matrix in $\mathbb{R}^{n \times d}$, where 
\begin{align}\label{eq:projection2}
 [P_{\mathbb{S}(n,d,\bar \gamma)}(X)]_i=P_{\mathbb{S}_{\gamma_i}^{d-1}}([X]_i), ~\forall i\!\in \! \{1,2,\dots,n\}.   
\end{align}
Thus, $P_{\mathbb{S}(n,d)}$ normalizes all row vectors w.r.t. the Euclidean norm, whereas $P_{\mathbb{S}(n,d,\bar \gamma_{l_1})}$ normalizes all row vectors w.r.t. the $l_1$-norm.

\section{Algorithm and problem formulation} \label{sec:algorithm}

We consider a multi-agent system with $n$ agents, each of which is assigned a unique index in the set $\{1, 2, \dots, n\}$. Each agent $i$ has a state $x_i(t)$ on a star boundary ${\mathbb{S}_{\gamma_i}^{d-1}}$ that evolves over discrete time steps $t \geq t_0$, where $t-t_0 \in \mathcal{N}_0 = \{0, 1, 2, \dots\}$. To simplify notation, we often omit expressing the explicit time dependence and write $x_i$ instead of $x_i(t)$. Let $\bar \gamma = (\gamma_1, \gamma_2, \ldots, \gamma_n)$  and let $X \in \mathbb{S}(n,d,\bar \gamma)$ be the matrix whose $i$-th row is $x_i$, where $i \in \{1, 2, \dots, n\}.$ The vector $x_i(t_0)$ is referred to as the initial state of agent $i$ and $X(t_0) = X_0$ is referred to as the initial state of the system.

\subsection{Algorithm}
Let $A = [a_{ij}] \in \mathbb{R}^{n \times n}$ be a nonnegative  weight matrix for a directed graph $\mathcal{G}(n)$, which represents interaction between agents. 
The $x_i$-states (equivalently $X$) are updated according to the algorithm 
\begin{align}
\label{eq:main_Algo_general}
X(t+1) & ~ = P_{\mathbb{S}(n,d,\bar \gamma)}(AX(t)),
\end{align}
where $t\geq t_0 \geq 0$.

The algorithm \eqref{eq:main_Algo_general} may also be written in a less compact form where the update of the state for each agent is clearly visible:   
\begin{align}\label{eq:main_Algo_xi}
x_i(t+1)  &~ = P_{\mathbb{S}_{\gamma_i}^{d-1}}([AX(t)]_i),
\end{align}
where $[AX(t)]_i=[A]_iX(t)=\sum_{j = 1}^n a_{ij}x_j(t)$ is the $i$-th row of $AX(t)$. In this formulation, each agent's state is updated based on a conical
combination of its neighbors' states, followed by a radial projection onto the agent-specific star boundary defined by the directional function $\gamma_i$.
Note, that every conical combination of states has an equivalent convex combination w.r.t. projection since $\frac{\alpha_i[A]_iX}{\|\alpha_i[A]_iX\|_2}=\frac{[A]_iX}{\|[A]_iX\|_2}$, where $\alpha_i=\frac{1}{[A]_i\bold{1}_{n,1}}>0$.

By using the definition of the radial projection in \eqref{eq:projection1},  we express the equation in \eqref{eq:main_Algo_xi} for $i \in \{1,2,\ldots, n\}$ as
\begin{equation}\label{eq:kalle1}
x_i(t+1) = \gamma\left(\frac{[AX(t)]_i}{\|[AX(t)]_i\|_2}\right) \frac{[AX(t)]_i}{\|[AX(t)]_i\|_2}.
\end{equation}
We see that it is required that $[AX(t)]_i \neq \bold{0}_{1,d}$. Lemma~\ref{lemma:well-defined} below shows that if $A \in \mathbb{R}^{n \times n}$ is a nonnegative matrix with diagonal elements being sufficiently large in comparison to the off-diagonal elements, then for any $X \in \mathbb{S}(n,d,\bar \gamma)$, we have $[AX]_i \neq \bold{0}_{1,d}$. So, if the condition on $A$ in Lemma~\ref{lemma:well-defined} is satisfied, no matter what $X(t_0) \in \mathbb{S}(n,d,\bar \gamma)$ we choose, there is no $t \geq t_0$ and $i \in \{1,2, \ldots, n\}$ such that $[AX(t)]_i = \bold{0}_{1,d}$.

\begin{lem}\label{lemma:well-defined}
Let $\bar \gamma = (\gamma_1, \gamma_2, \ldots, \gamma_n)$, where $\gamma_i$ is a directional function on $\mathbb{S}^{d-1}$ for each $i$. If the nonnegative matrix $A = [a_{ij}] \in \mathbb{R}^{n \times n}$ satisfies
\begin{align}
\label{eq:lem1}
    a_{ii}\!\min_{{x} \in \mathbb{S}^{d-1}}\! \gamma_i (x)> \sum_{j \neq i}a_{ij} \max_{{x} \in \mathbb{S}^{d-1}} \!\gamma_j (x)  \; \;\forall i\!\in \! \{1,2,\dots,n\},
\end{align}
then $[AX]_i \neq \bold{0}_{1,d}$ for $i \in \{1,2,\dots,n\}$, $X \in \mathbb{S}(n,d,\bar \gamma)$.
\end{lem}
\textbf{Proof:} 
Let $X \in \mathbb{S}(n,d,\bar \gamma)$ and $x_i = [X]_i \in \mathbb{S}_{\gamma_i}^{d-1}$ for each $i$. Assume $\tilde i$ is such that $[AX]_{\tilde i} = a_{\tilde i \tilde i}x_{\tilde i}+\sum_{j \neq \bar \tilde i}a_{\tilde i j}x_j = \bold 0_{1,d}$. Then 
$a_{\tilde i \tilde i}\|x_{\tilde i}\|_2=\|\sum_{j \neq \tilde i}a_{\tilde i j}x_j\|_2 \leq  \sum_{j \neq \tilde i}a_{\tilde i j}\|x_j\|_2$. However, $\|x_i\|_2 = \gamma_i(x_i)$ for each $i$. Hence, $a_{\tilde i \tilde i}\gamma_{\tilde i} (x_{\tilde i}) \leq \sum_{j \neq \tilde i}a_{\tilde i j}\gamma_j (x_j)$, which implies 
\begin{equation}\label{eq:kalle2}
a_{\tilde i \tilde i}\min\limits_{{x} \in \mathbb{S}^{d-1}} \gamma_{\tilde i} (x) \leq \sum_{j \neq \tilde i}a_{\tilde i j} \max\limits_{{x} \in \mathbb{S}^{d-1}} \gamma_j (x).
\end{equation}
The inequality in \eqref{eq:kalle2} cannot be satisfied if the strict inequality in \eqref{eq:lem1} is satisfied. \hfill $\qed$

Note that in Lemma~\ref{lemma:well-defined} the requirement on $A$ is that it is nonnegative; we do not require it to be, for example, a nonnegative weight matrix for a strongly connected graph.  
If we assume $A$ to be a nonnegative weight matrix for the directed graph $\mathcal{G}(n) = (\mathcal{V}(n), \mathcal{E})$ that satisfies \eqref{eq:lem1}, this means that $(i,i) \in \mathcal{E}$ for all $i$, i.e., all self-loops are contained in the directed graph. If all $n$ agents are constrained to move on the same Euclidean sphere, i.e., all the $\gamma_i$'s are equal and constant, the condition in \eqref{eq:lem1} reduces to $a_{ii} > \sum_{j \neq i}a_{ij}$, i.e., the nonnegative weight matrix $A$ should be diagonally dominant. 

\subsection{Problem formulation}\label{sec:problem}
Now, the overall problem is to determine conditions under which asymptotic consensus is achieved for the algorithm \eqref{eq:main_Algo_general}, when $A$ is a nonnegative weight matrix for the directed graph $\mathcal{G}(n) = (\mathcal{V}(n), \mathcal{E})$ that, together with $\bar \gamma$, satisfies \eqref{eq:lem1}. More specifically, this concerns conditions on the initial states under which the following types of asymptotic consensus are achieved. 

Asymptotic consensus of directions means that for all $i,j$ 
\begin{equation}\label{eq:5}
    \left \|\frac{x_i(t)}{\|x_i(t)\|_2} - \frac{x_j(t)}{\|x_j(t)\|_2}\right \|_2 \rightarrow 0 \text { as } t \rightarrow +\infty.
\end{equation}
Asymptotic consensus of directions with point-wise convergence means that there exists $u \in \mathbb{S}^{d-1}$, such that 
\begin{equation}\label{eq:nisse:2}
    \left \|\frac{x_i(t)}{\|x_i(t)\|_2} - u \right \|_2 \rightarrow 0 \text { as } t \rightarrow +\infty.
\end{equation}

By using matrix-notation, asymptotic consensus of directions with point-wise convergence implies that there exists a row-vector $u_c$ with nonzero entries such that
\begin{align} \label{eq:def_consensus}
X_{\infty}=u_c^T u,
\end{align}
where $X_{\infty}=\lim\limits_{t \rightarrow \infty} X(t)$.
When $\gamma_i = \gamma_j$ for all $i,j$, there exists a scalar $\kappa \neq 0$ such that $u_c = \kappa\boldsymbol{1}_{1,n}$, and hence all the $x_i$'s asymptotically converge to the same constant point, which is a scaling of $u$. This would be the situation when all the states of the agents are on the same star boundary. We then say that the algorithm exhibits asymptotic consensus of states with point-wise convergence.

\subsection{Only relative information}

A special case of algorithm \eqref{eq:main_Algo_general} is when $\gamma_i(x) = \gamma_j(x) = \frac{1}{\|x\|_p}$ for all $i,j$. In this case, \eqref{eq:kalle1} can be written as 
\begin{align}\label{eq:main_Algo_Lp}
x_i(t+1) = \frac{\sum_{j = 1}^n a_{ij}x_j(t)}{\|\sum_{j = 1}^n a_{ij}x_j(t)\|_p},
\end{align}
for every $ i \in \{1,2, \ldots, n\}$, $t\geq t_0 \geq 0$. 

Algorithm \eqref{eq:main_Algo_Lp} is presented in a global reference frame. This means that $x_i$ is the position in a global coordinate system. However, for the  Euclidean unit sphere (i.e., the set $\mathbb{S}^{d-1}$) it can be implemented in a distributed fashion by using only local and relative information. 
Indeed, suppose that at each time $t$, each agent observes the unit vectors of its neighbors expressed in its own coordinate system, where the each agent's coordinate system is expressed as an orthogonal transformation of the global coordinate system. The unit vector $x_i$, expressed in the coordinate system of agent $i$, is given by 
$\tilde x_i = x_iR_i$, where $R_i$ is the orthogonal matrix that rotates the coordinate system of agent $i$ to the global coordinate system. The unit vector $\tilde x_j$ can be expressed in the coordinate system of agent $i$ as $\tilde x_{ij}$ using the formula $\tilde{x}_{ij} = \tilde x_{j}R_{ij}^T$, where $R_{ij} = R_i^TR_j$ is the rotation matrix that rotates the coordinate system of agent $j$ to that of agent $i$. 

The $R_i$'s and the $R_{ij}$'s could be static or dynamic relative to the global coordinate system. Using relative information only, the algorithm \eqref{eq:main_Algo_Lp} is expressed as 
 \begin{equation}\label{eq:mainAlg2}
 \tilde x_i(t+1) = \frac{\sum_{j = 1}^n a_{ij}\tilde x_{ij}(t)}{\|\sum_{j = 1}^n a_{ij}\tilde x_{ij}(t)\|_2}, 
 \end{equation}
for every $ i \in \{1,2, \ldots, n\}$, $t\geq t_0\geq 0$.

\subsection{Connection to discrete-time linear consensus}\label{subsec:linear_consensus}

Consider the linear discrete-time consensus algorithm \begin{align} \label{eq:standard_cons}
    X(t+1) = AX(t),
\end{align}
where $A$ is such that asymptotic consensus of states with point-wise convergence is ensured. In \eqref{eq:standard_cons}, the row-wise radial projection of $X(t)$ onto $\mathbb{S}(n,d,\bar\gamma)$ is not present as it is in \eqref{eq:main_Algo_general} and
it is not guaranteed that $X(t) \in \mathbb{S}(n,d,\bar\gamma)$ for all $t$. If the graph $\mathcal{G}(n)$ is symmetric (or undirected) and connected, a choice of $A$ is $A = I_n - \eta L$, where 
$L$ is the graph Laplacian matrix for $\mathcal{G}(n)$ and $\eta > 0$. If $\eta$ is chosen sufficiently small, $A$ satisfies the conditions in Lemma~\ref{lemma:well-defined} while asymptotic consensus of states with point-wise convergence is ensured for algorithm \eqref{eq:standard_cons}. On the other hand, if a matrix $A$ satisfies the conditions in Lemma~\ref{lemma:well-defined}, so does $\alpha A$ where $\alpha > 0$. Furthermore, we may equivalently use $\alpha A$ instead of $A$ in \eqref{eq:main_Algo_general}, since  $P_{\mathbb{S}(n,d,\bar \gamma)}(\alpha AX(t)) = P_{\mathbb{S}(n,d,\bar \gamma)}(AX(t))$. However, for large enough $\alpha$, using $\alpha A$ instead of $A$ in \eqref{eq:standard_cons} does not ensure point-wise convergence since the system becomes unstable. 

\section{Preliminary results}\label{sec:convergence_matrices}
The section provides tools for showing the main results of the paper presented in Section~\ref{sec:convergence_to_consensus}.
Throughout we assume $A$ is a nonnegative  weight matrix for the directed graph $\mathcal{G}(n) = (\mathcal{V}(n), \mathcal{E})$, which, together with $\bar \gamma$, satisfies \eqref{eq:lem1}. 

For such an $A$, if $X(t)$ is generated by \eqref{eq:main_Algo_general} for $t \geq t_0 \geq 0$, we may write  
\begin{align} \label{eq:easier_form1:pre}
        X(t)\!=\!\tilde A(k,t_0,X(t_0))X(t_0),
    \end{align}
where $k = t - t_0$. This way of representing $X(t)$ will be formally introduced in Section~\ref{sec:convergence_to_consensus} for showing  asymptotic consensus of states or directions. 

Each $\tilde A(k,t_0,X(t_0))$-matrix, where $k > 0$, is expressed as the product of $k$ matrices, each of which comprises the product of a state-dependent diagonal matrix and the matrix $A$.
As will be shown in Section~\ref{sec:convergence_to_consensus}, the $\tilde A(k,t_0,X(t_0))$-matrices have the property that, for any pair of matrix rows, the ratio of their norms is bounded from below and above by positive constants independent of $k$, $t_0$, and $X(t_0)$. To establish this result in Section~\ref{sec:convergence_to_consensus}, we rely on Proposition~\ref{corollary:1} in this section.

In this section, we study products for a time-varying $A(t)$-matrix, given as $\tilde{A}(k,t) = A(t_0 + k-1)A(t_0 +k-2)\cdots A(t_0+1)A(t_0)$. As seen, those products do not involve states generated by \eqref{eq:main_Algo_general}. Proposition~\ref{corollary:1} establishes that, under certain assumptions, such product matrices have the property that, for any pair of matrix rows, the ratio of their norms is bounded from below and above by positive constants independent of $k$ and $t_0$. 

We then show that for such product matrices, asymptotic consensus of directions with point-wise linear convergence for the matrix rows is achieved, see Proposition~\ref{prop:convergence_Atilde}. This latter result is then used in Section~\ref{sec:convergence_to_consensus} to show an equivalent result for the  $A(k,t_0,X(t_0))$-matrices, which in turn can be used to show asymptotic consensus of directions or states for the $X(t)$-matrices.

\begin{assumption}\label{ass:new:1}
$A(t) \in \mathbb{R}^{n \times n}$ is a time-varying matrix. There is a strongly connected graph $\mathcal{G}(n)$ such that $A(t)$ is a nonnegative weight matrix with positive diagonal for all $t \geq 0$. The positive elements of $A(t)$ are bounded from below and above by the positive constants $\sigma_l$ and $\sigma_u$, respectively, for all $t \geq 0$.
\end{assumption}

Suppose $A(t)$ satisfies Assumption~\ref{ass:new:1}. The fact that there is a strongly connected graph for which each $A(t)$ is a nonnegative weight matrix with positive diagonal ensures that $A(t)$ is irreducible for each $t$. This implies each $A(t)$ contains no zero-rows or zero-columns. Furthermore, since $A(t)$ is a nonnegative matrix with $[A(t)]_{ii}~>~0$ for all $i$, it is primitive.

\begin{definition}\label{def:111}
  For a time-varying $A(t)$, where $A(t) \in \mathbb{R}^{n \times n}$ for $t \geq 0$, 
  we define the matrix $\tilde A(k, t_0)\in \mathbb{R}^{n \times n}$ as:
\begin{align}\label{eq:definition_A_tilde}
 &  \tilde A(k, t_0) = \prod_{l = 0}^{k-1} A(t_0 +l) \\
 \nonumber
=& A(t_0 + k - 1 )A(t_0 + k-2) \cdots A(t_0),  
\end{align}
for integers $k \geq 1$ and $t_0 \geq 0$. We further define $\tilde A(0,t_0) = I_n.$ 
\end{definition}
 
Suppose $A(t)$ satisfies Assumption~\ref{ass:new:1}. First, we note that $\tilde A(k, t_0)$ in Definition~\ref{def:111} is a nonnegative matrix for all $k$ and $t_0$ since it is a product of nonnegative matrices. However, given that Assumption~\ref{ass:new:1} is satisfied, we cannot simply conclude that $\|\tilde{A}(k,t_0)\|$ converges to a value as $k$ goes to infinity, nor can we simply exclude the possibility that $\lim\limits_{k \rightarrow \infty}\|\tilde{A}(k,t_0)\| \in \{0, \infty\}$, where $\|\cdot \|$ denotes some matrix norm. What can be concluded, as stated in Proposition~\ref{corollary:1} below, is that for any pair of rows of $\tilde{A}(k,t_0)$, the ratio between their norms is bounded by constants independent of $k$ and $t_0$.

\begin{proposition} \label{corollary:1}

Suppose $A(t)$ satisfies Assumption~\ref{ass:new:1}. There exist positive constants $\delta_{1,p}$ and $\delta_{2,p}$ that depend on $\sigma_l$ and $\sigma_u$ only (and not on $k$), such that
\begin{equation}
\label{eq:bounds:rows_lp}
\delta_{1,p} \leq \frac{\|[\tilde A(k,t_0)]_{i}\|_p}{\|[\tilde A(k,t_0)]_{j}\|_p} \leq \delta_{2,p},
\end{equation}  
for all $t_0 \geq 0$,  $k \geq 1$, $i \in \{1,2, \ldots, n\}$, $j \in \{1,2, \ldots, n\}$, and $p \in \mathcal{N} \cup \{\infty\}$ where $\mathcal{N}= \{1, 2, \ldots \}$. 
\end{proposition}

\textbf{Proof:} 
If $y \in \mathbb{R}^n$, and $q_1$ and $q_2$ are positive integers or equal to $\infty$, such that $1\leq q_1 \leq q_2$, then $\|y\|_{q_1} \leq n^{(\frac{1}{q_1} - \frac{1}{q_2})}\|y\|_{q_2}$, where $\frac{1}{\infty}: = 0$. Thus, if the statement in Proposition~\ref{corollary:1} is true for $p= 1$, the statement is true for any $p \in \mathcal{N} \cup \{\infty\}$. Thus, we continue to prove the statement for $p=1$. 

Due to Assumption~\ref{ass:new:1}, $\sigma_lI_n \leq A(t) \leq \sigma_u \boldsymbol 1_{n,n}$. By using these two inequalities and the definition of $\tilde A(k,t_0)$ \eqref{eq:definition_A_tilde}, we obtain the inequalities
\begin{equation}
\label{eq:bounds_norm1}
\sigma_l^{k} \leq \|[\tilde{A}(k,t_0)]_{i}\|_1 \leq \sigma_u^k n^{k} \text{   for all }i, k, t_0 \geq 0. 
\end{equation}
Thus,
\begin{align}
\label{eq:bounds:new1}
&\frac{\sigma_l^{k}}{\sigma_u^k n^{k}} \leq \frac{\|[\tilde A(k,t_0)]_{i}\|_1}{\|[\tilde A(k,t_0)]_{j}\|_1} \leq \frac{\sigma_u^k n^{k}}{\sigma_l^{k}} \text{ for all } i,j,k, t_0 \geq 0.
\end{align}
If $1 \leq k \leq n$, then $\sigma_l^{n}/(\sigma_u^{n} n^{n})$ and $(\sigma_u^{n} n^{n})/\sigma_l^{n}$ are a lower and an upper bound, respectively, on $\|[\tilde A(k,t_0)]_{i}\|_1/\|[\tilde A(k,t_0)]_{j}\|_1$ for all $i$, $j$, $t_0 \geq 0$. 

It remains to consider the case $k > n$. By using the definition of $\tilde{A}(k,t_0)$ \eqref{eq:definition_A_tilde}, we may write 
\begin{equation}\label{eq:spliting_A_tilde}
\tilde A(k, t_0) = \tilde{A}(n, t_0+k-n)\tilde{A}(k-n,t_0). 
\end{equation}
By using \eqref{eq:spliting_A_tilde} and the fact that $A(t)$ is nonnegative for all $t$, we obtain:
\begin{align}
\nonumber
& \|[\tilde{A}(k,t_0)]_i\|_1
= \|[\tilde{A}(n, t_0+k-n)\tilde{A}(k-n,t_0)]_i\|_1 \\
\label{eq:nils:1}
=~& \sum_{j=1}^n [\tilde{A}(n, t_0+k-n)]_{ij}\|[\tilde{A}(k-n,t_0)]_j\|_1.
\end{align}
Let 
\begin{equation}\label{eq:beta}
    \beta(k,t_0) = \max\limits_{1 \leq j \leq n}\left (\|[\tilde{A}(k-n,t_0)]_j\|_1\right ),
\end{equation}
where we know that $\sigma_l^{k-n} \leq \beta(k, t_0) \leq \sigma_u^{k-n} n^{k-n}$ due to \eqref{eq:bounds_norm1}.

Let $\bar A$ be the binary nonnegative weight matrix with positive diagonal for the strongly connected graph $\mathcal{G}(n)$. 
It holds that 
\begin{equation}\label{eq:lemma1_ineq_Atilda}
\sigma_l^n \bar A^n \leq \tilde{A}(n, t_0+k-n) \leq \sigma_u^n \bar A^n.
\end{equation}
 Since $\bar A$ is a nonnegative, irreducible and binary matrix with positive elements on the diagonal, $\bar A^n$ is a positive matrix where the smallest element thereof is greater than or equal to $1$, see~\citep{horn2012matrix} for further details. However, the elements of $\bar A^n$ must also be smaller than or equal to $n^{n-1}$ since $\bar A^n \leq \bold{1}_{n,n}^n$. From these results together with \eqref{eq:lemma1_ineq_Atilda}, we obtain 
\begin{equation}\label{eq:lemma1_ineq_Atilda2}
\sigma_l^n \leq [\tilde{A}(n, t_0+k-n)]_{ij} \leq \sigma_u^n n^{n-1},
\end{equation}
for all $i,j$.

By using \eqref{eq:beta} and \eqref{eq:lemma1_ineq_Atilda2} in \eqref{eq:nils:1}, we get that
\begin{equation}\label{eq:nils2}
 \sigma_l^{n}\beta(k,t_0) \leq \|[\tilde{A}(k,t_0)]_i\|_1 \leq  \sigma_u^n n^{n}\beta(k,t_0),
\end{equation}
for all $i$. 
From \eqref{eq:nils2} we obtain for $k > n$ that
\begin{align}
\label{eq:nils3}
&\frac{\sigma_l^{n}}{\sigma_u^n n^{n}} \leq \frac{\|[\tilde A(k,t_0)]_{i}\|_1}{\|[\tilde A(k,t_0)]_{j}\|_1} \leq \frac{\sigma_u^n n^{n}}{\sigma_l^{n}} \text{ for all } i,j.
\end{align}
Thus, all cases $k \geq 1$ are covered, yielding 
\begin{equation}
    \delta_{1,1} = \frac{\sigma_l^{n}}{\sigma_u^n n^{n}}, \quad \text{ and } \quad \delta_{2,1} = \frac{\sigma_u^n n^{n}}{\sigma_l^n}.
\end{equation}
~$\hfill$ $\qed$

Now we provide the main result of the section, which can be formulated as follows: there is asymptotic consensus of directions for the $\tilde A(k,t_0)$-matrices with a linear convergence rate independent of the choice of $t_0 \geq 0$. 

\begin{proposition}\label{prop:convergence_Atilde}
For $A(t)$ satisfying Assumption~\ref{ass:new:1}, there exists $$v: \mathcal{N}_0 \rightarrow \mathbb{S}^{n-1},$$ such that each
sequence $\left \{ \frac{[\tilde{A}(k,t_0)]_i}{\|[\tilde{A}(k,t_0)]_i\|_2} \right \}_{k = 1}^{\infty}$ for all $i\in\{1,2,\dots,n\}$ converges linearly to the same vector $v(t_0)$. 
\end{proposition}

\textbf{Proof:}
We define for $k \geq 1$
\begin{align} \label{eq:Y_kt}
Y(k,t_0) = P_{\mathbb{S}(n,n)}(\tilde{A}(k,t_0)),
\end{align}
which is the row-normalized $\tilde{A}(k,t_0)$-matrix with respect to the Euclidean norm. 
Then $Y(k+l, t_0)=P_{\mathbb{S}(n,n)}(\tilde{A}(k+l,t_0))$ for $l\geq 0$  can be expressed as
\begin{equation}\label{eq:nisse3}
Y(k+l, t_0) = P_{\mathbb{S}(n,n)}(B(l, k + t_0)Y(k,t_0)),
\end{equation}
where $B(l,k + t_0) = P_{\mathbb{S}(n,n,\bar \gamma_{l_1})}(\tilde A(l,k + t_0)V(k,t_0))$, and $V(k,t_0)$ is a positive diagonal matrix with the $i$'th diagonal element given by $\frac{\|[\tilde{A}(k,t_0)]_i\|_2}{\|[\tilde{A}(k,t_0)]_1\|_2}$. Note that $V(k,t_0)$ is a bounded matrix for all $t_0 \geq 0$, $k \geq 1$ according to Proposition~\ref{corollary:1}.
Equation \eqref{eq:nisse3} holds, since 
\begin{align}
    \nonumber
    & P_{\mathbb{S}(n,n)}(B(l,k + t_0)Y(k,t_0))\\
    \nonumber
    =~& P_{\mathbb{S}(n,n)}(P_{\mathbb{S}(n,n,\bar \gamma_{l_1})}(\tilde A(l,k+t_0)V(k,t_0))Y(k,t_0)) \\
    \nonumber
    =~& P_{\mathbb{S}(n,n)}\left (\frac{\tilde A(l,k+t_0)\tilde A(k,t_0)}{\|[\tilde{A}(k,t_0)]_1\|_2} \right) \\
    =~&  P_{\mathbb{S}(n,n)}(\tilde A(k+l, t_0)) = Y(k+l,t_0).
\end{align}

By definition, $Y(k,t_0) \in \mathbb{S}(n,n)$ is nonnegative for all $t_0$.
Since $B(l,k+t_0)$ is a stochastic matrix,
the minimum cone of the matrix $Y(k+1,t_0)$ is a subset of the minimum cone of $Y(k,t_0)$, i.e., $\mathcal{C}(Y(k+1,t_0)) \subset \mathcal{C}(Y(k,t_0))$, which means that the maximum angle between rows is decreasing as a function of $k$.

In the following, we show the stronger result that the maximum angle between the rows of $Y(k,t_0)$ decreases to zero as $k \rightarrow \infty$. Consequently, all the rows of $Y(k,t_0)$ converge to the same vector. Our proof considers $l = n$ and $k = sn$ for the integers $s \geq 1$. In between these time steps, the maximum angle between rows is not increasing due to the cone invariance property discussed above.

In the following, we set $l=n$. There are positive constants $\tilde{\delta}_1$ and $\tilde{\delta}_2$ that only depend on $\sigma_l$ and $\sigma_u$ from Assumption~\ref{ass:new:1} such that
\begin{align} \label{ineq:bounds_B}
 0 < \tilde{\delta}_1 \leq [B(n,k + t_0)]_{ij} \leq \tilde{\delta}_2 < 1  
\end{align} holds for any $t_0 \geq 0$, $k \geq n$, and all $i$, $j$. These bounds are assured by the following. The diagonal elements of the diagonal matrix $V(k + t_0)$ are uniformly bounded from above and below by positive constants, expressed in terms of $\sigma_l$ and $\sigma_u$, according to Proposition~\ref{corollary:1}. The elements of $\tilde{A}(n,k+t_0)$ are also uniformly bounded by such constants; see \eqref{eq:lemma1_ineq_Atilda2} in the proof of Proposition~\ref{corollary:1}. Hence, the constants we seek, $\tilde \delta_1$ and $\tilde \delta_2$, exist for the matrix $B(n,k + t_0) = P_{\mathbb{S}(n,n,\bar \gamma_{l_1})}(\tilde A(n,k + t_0)V(k,t_0))$.

Let us now consider the time-evolution of $Z(s) = Y(sn, t_0)$: 
\begin{align}\label{eq:simpler_form}
Z(s + 1) = P_{\mathbb{S}(n,n)}(\bar B(s) Z(s)),
\end{align}
where $\bar{B}(s) = B(n,sn + t_0)$  and $Z(s) = Y(sn, t_0) \in  \mathbb{S}(n,n)$ is positive for all $s \geq 1$ and $t_0 \geq 0$. The matrix $\bar B(s)$ is row-stochastic with all its elements being positive and bounded from below and above by $\tilde \delta_1$ and $\tilde \delta_2$, respectively.

For the maximum angle $\phi_s=\phi_{max}(Z(s))$ (see \eqref{eq:phi_max} for the definition) between rows of the matrix $Z(s)$, we define 
\begin{align}
 \mathcal{J}_Z(s) = \{(i,j) : [Z(s)]_i[Z(s)]_j^T = \cos(\phi_s)\}.   
\end{align}The function $\phi_{\max}(\cdot)$ is continuous on $\mathbb{S}(n,n)$. The set $\mathcal{J}_Z(s)$ is the set of index pairs for which the inner product of the corresponding unit rows of $Z(s)$ equals $\cos(\phi_s)$. Now, consider an arbitrary index pair $(\bar i,\bar j)$. By using \eqref{eq:simpler_form}, we can write
\begin{align} \label{eq:ineq_inner_product_Z(s+1)}
    & [Z(s+1)]_{\bar i}[Z(s+1)]_{\bar j}^T\\ \nonumber 
    =&  \frac{\sum_{(i,j)} [\bar B(s)]_{\bar ii}[\bar B(s)]_{\bar jj}[Z(s)]_{i}[Z(s)]_{j}^T}{\|\sum_{i = 1}^n[\bar B(s)]_{\bar ii}[Z(s)]_{i}\|_2\|\sum_{j = 1}^n[\bar B(s)]_{\bar jj}[Z(s)]_{j}\|_2}. \nonumber
\end{align}
Here, the notation $\sum_{(i,j)}$ denotes summing over all pairs of indices $i$ and $j$. 

Since $Z(s)$ is a positive matrix for all $s \geq 1$, it follows that
\begin{align} \label{eq:ineq_ZiZj}
0 < \cos(\phi_s)\leq [Z(s)]_i[Z(s)]_j^T\leq 1,
\end{align}
for any pair of indices $(i,j)$.
Next, we notice that since $\bar B(s)$ is a stochastic matrix, it holds that $\sum_{(i,j)}[\bar B(s)]_{\bar ii}[\bar B(s)]_{\bar jj}=1$. By combining these two observations, we conclude that the numerator of the right-hand side of \eqref{eq:ineq_inner_product_Z(s+1)} is bounded from below by $\cos(\phi_s)$.

Now we take a closer look at the denominator. We observe that
\begin{align}
\label{eq:kalle:100}
     &\bigg\|\sum_{i = 1}^n [\bar B(s)]_{\bar ii} [Z(s)]_i\bigg\|_2^2\\ \nonumber
    & =\sum_{(i,j)} [\bar B(s)]_{\bar ii} [\bar B(s)]_{\bar ij} [Z(s)]_i [Z(s)]_j^T.
\end{align}
We can split the sum in \eqref{eq:kalle:100} into two parts: one over pairs of indices not in $\mathcal{J}_Z(s)$ and the other over pairs of indices in $\mathcal{J}_Z(s)$.  By using \eqref{eq:ineq_ZiZj}, we obtain
\begin{align}
    \nonumber
    & \left\|\sum_{i = 1}^n [\bar B(s)]_{\bar ii} [Z(s)]_i\right\|_2^2  \\
    &\leq \sum_{(i,j) \notin \mathcal{J}_Z(s)}[\bar B(s)]_{\bar ii}[\bar B(s)]_{\bar ij} \cdot 1 \; \\
    &\quad  +\sum_{(i,j)\in \mathcal{J}_Z(s)}[\bar B(s)]_{\bar i i}[\bar B(s)]_{\bar i j }\cos( \phi_s). \nonumber
\end{align}
By adding and subtracting the same term
\begin{equation}
    \sum_{(i,j) \in \mathcal{J}_Z(s)}[\bar B(s)]_{\bar ii}[\bar B(s)]_{\bar ij},
\end{equation}
we obtain the following:
\begin{align}
\nonumber
    & \bigg\|\sum_{i = 1}^n [\bar B(s)]_{\bar ii} [Z(s)]_i\bigg\|_2^2 \leq \sum_{(i,j)}[\bar B(s)]_{\bar ii}[\bar B(s)]_{\bar ij} \\
    \nonumber
    + & \sum_{(i,j)\in \mathcal{J}_Z(s)}\!\![\bar B(s)]_{\bar i i}[\bar B(s)]_{\bar i j }( \cos(\phi_s)-1) \\
    \leq~& 1 - \tilde \delta_1^2(1 - \cos(\phi_s)),
\end{align}
where $\tilde \delta_1$ is the lower bound for elements of the matrix $\bar B(s)$; see \eqref{ineq:bounds_B}. 
Consequently, we establish the following bound for the denominator of \eqref{eq:ineq_inner_product_Z(s+1)}:
\begin{align}
    \nonumber
    & \bigg\|\sum_{i = 1}^n[\bar B(s)]_{\bar ii}[Z(s)]_{i}\bigg\|_2 \bigg\|\sum_{j = 1}^n[\bar B(s)]_{\bar jj}[Z(s)]_{j}\bigg\|_2 \\
    \leq~& 1 - \tilde \delta_1^2(1 - \cos(\phi_s)). 
\end{align}
Since $(\bar i,\bar j)$ was arbitrary, we can now conclude that  \begin{align}\label{eq:ineq_phi}
    & \cos(\phi_{s+1}) \geq \frac{\cos(\phi_s)}{1 - \tilde\delta_1^2(1 - \cos(\phi_s))} \geq \cos(\phi_s).
\end{align} 
The sequence $\{\cos(\phi_s)\}$ is bounded and non-decreasing. Therefore, by the monotone convergence theorem \citep{rudin1964principles}, it converges to some limit $L$ where $0 \leq L \leq 1$. Furthermore, by applying the ``squeeze theorem'' to \eqref{eq:ineq_phi}, we conclude that $L$ must be either $0$ or $1$.
Since $\cos(\phi_s)>0$ \eqref{eq:ineq_ZiZj} at $s=1$ and remains positive for all $s>1$, it follows that the sequence $\{\cos(\phi_s)\}$ converges to $L=1$ as $s \to \infty$. This implies that the largest angle between any pair of rows of the matrix $Z(s)$ converges to $0$.
Since $\mathcal{C}(Y(k+1,t_0)) \subset \mathcal{C}(Y(k,t_0))$, the angle between any pair of rows of $Y(k,t_0)$ also goes to $0$ as $k \to \infty$. Consequently, there exists a unit vector $v(t_0)$ to which all the rows of $Y(k,t_0)$ converge. 

By using \eqref{eq:ineq_phi}, we can also deduce the convergence rate. We define   
\begin{equation}
    e_s = 1 - \cos(\phi_s) = \max_{i,j}\frac{\|Z_{i}(s) - Z_{j}(s)\|_2^2}{2}.
\end{equation}
The Taylor expansion of
\begin{equation}
    g(e_s)= \frac{1-e_s}{1 - \tilde\delta_1^2e_s} = \frac{\cos(\phi_s)}{1 - \tilde\delta_1^2(1 - \cos(\phi_s))}
\end{equation}
at the point $0$ is given by 
\begin{equation}
    g(e_s)=1-(1 - \tilde \delta_1^2)e_s+O(e_s^2).
\end{equation}
Thus, there must be an $s_0$ such that for all $s \geq s_0$, it holds that 
$1 - g(e_s) \leq (1 - \frac{\tilde \delta_1^2}{2})e_{s}$. Now, by combining this result with \eqref{eq:ineq_phi}, we conclude that
\begin{equation}
    e_{s+1}\leq (1-\frac{\tilde \delta_1^2}{2}) e_{s},
\end{equation}
for $s \geq s_0$, 
which ensures that the largest distance between rows of $Z(s)$ converges linearly to zero. Moreover, since $\mathcal{C}(Y(k+1,t_0)) \subset \mathcal{C}(Y(k,t_0))$, it follows that $\sqrt{2e_s}$ upper bounds the maximum distance between any row of $Y(k,t_0)$ and $v(t_0)$ for $k \geq n$. 
\hfill \qed

\section{Main results}\label{sec:convergence_to_consensus}
In this section, we present the main results on the convergence for the algorithm \eqref{eq:main_Algo_general}.
Using Proposition \ref{prop:convergence_Atilde}, we provide a necessary and sufficient condition for asymptotic consensus of directions. As this condition is hard to verify explicitly, we further provide additional assumptions under which the condition is guaranteed to be fulfilled.

We first start by noting that  \eqref{eq:main_Algo_general} can be written in the following alternative form
\begin{align} \label{eq:easier_form1}
        X(t_0+k)=\tilde A(k,t_0,X_0)X_0,
    \end{align}
where $X_0=X(t_0)$, $k > 0$, and 
\begin{align}
\label{eq:olle2000}
    &\tilde  A(k,t_0,X_0)=\prod_{m = 0}^{k-1}\left ( D_{\bar \gamma}(AX(t_0+m)) A \right )\\
    & =D_{\bar \gamma}(AX(t_0 + k - 1 ))A \cdots D_{\bar \gamma}(AX(t_0))A, \nonumber
\end{align}
where for a matrix $B$ with non-zero rows
\begin{align} \label{D_bar_gamma}
D_{\bar \gamma}(B)\!=\!\text{diag}\Bigg(\frac{\gamma_1\left(\!\frac{[B]_1}{\|[B]_1\|_2}\!\right)} {\|[B]_1\|_2}, \dots, \frac{\gamma_n\left(\!\frac{[B]_n}{\|[B]_n\|_2}\!\right)} {\|[B]_n\|_2}\Bigg).
\end{align}
This holds since $P_{\mathbb{S}(n,d,\bar \gamma)}(AX) = D_{\bar \gamma}(AX)AX$. 

\begin{assumption}\label{ass:100}
$\bar \gamma = (\gamma_1, \gamma_2, \ldots, \gamma_n)$ is such that $\gamma_i$ is a directional function on $\mathbb{S}^{d-1}$ for each $i$. $A = [a_{ij}]$ is a nonnegative weight matrix for a strongly connected graph $\mathcal{G}(n)$. 
The matrix $A = [a_{ij}] \in \mathbb{R}^{n \times n}$ satisfies
\begin{align}
\label{eq:assumption_2}
    a_{ii}\min_{{x} \in \mathbb{S}^{d-1}} \gamma_i (x)> \sum_{j \neq i}a_{ij} \max_{{x} \in \mathbb{S}^{d-1}} \gamma_j (x)  \; \;\forall i \in \{1,...,n\}.
\end{align}
\end{assumption}

\begin{proposition}~\label{prop_new:3}
Suppose the matrix $A$ together with $\bar \gamma$ satisfies Assumption~\ref{ass:100} and $X(t)$ is given by \eqref{eq:main_Algo_general}. Let $\tilde A(k,t_0,X_0)$ be given by \eqref{eq:olle2000}. Then there exists a continuous function
\begin{equation}
v: \mathbb{S}(n,d, \bar \gamma) \rightarrow (R^+)^n \cap \mathbb{S}^{n-1},
\end{equation}
such that each sequence $\left \{\frac{[\tilde{A}(k,t_0,X_0)]_i}{\|[\tilde{A} (k,t_0,X_0)]_i\|_2}\right\}_{k=1}^{\infty}$ for all $i\in\{1,2,\dots,n\}$ converges linearly to the same vector $v(X_0)$.
\end{proposition}

\textbf{Proof:}
Given that Assumption~\ref{ass:100} is satisfied, Lemma~\ref{lemma:well-defined} applies. Thus, 
\begin{equation}
    X \mapsto AX \mapsto [D_{\bar \gamma}(AX)]_{ii},
\end{equation}
is a continuous function from the compact set $\mathbb{S}(n,d, \bar \gamma)$ to $R^+$ for each $i$. This means that the range of the map for each $i$ is bounded from below and above by two positive constants. This, in turn, means that all the $n$ diagonal elements of $D_{\bar \gamma}(AX)$ are bounded from above and below by two positive constants. Hence $A(m) = D_{\bar \gamma}(AX(t_0+m))A$ satisfies Assumption~\ref{ass:new:1}, and Proposition~\ref{prop:convergence_Atilde} ensures that there is a unit vector $v(X_0)$ such that the sequence in question converges linearly to $v(X_0)$. The rate of convergence may be expressed only by the constants bounding the map above together with $\sigma_l$ and $\sigma_u$ from Assumption~\ref{ass:new:1}. These constants are independent of specific choices of $i$ and $X_0$. 

Now we show that $v(X_0)$ is continuous in $X_0$. At this point we know that for a suitable matrix norm $\|\cdot \|$ (such as the Frobenius norm) and any $X_0 \in \mathbb{S}(n,d, \bar \gamma)$, for  $\epsilon > 0$ there is a positive integer $T(\epsilon)$ such that for $k \geq T(\frac{\epsilon}{3})$, it holds 
\begin{align}
\|P_{\mathbb{S}(n,n)}(\tilde A(k,t_0,X_0)) -  \mathbf{1}_{n,1} v(X_0))\| \leq \frac{\epsilon}{3}.   
\end{align}  
The matrix function $P_{\mathbb{S}(n,n)}(\tilde A(k,t_0,X_0))$ is continuous with respect to $X_0$ for every $k > 0$. Thus, for $\epsilon > 0$ there must be $\delta(\epsilon) > 0$ such that $\|\tilde X_0 - \bar X_0\| \leq \delta$ implies 
\begin{align}
\nonumber
    & \|P_{\mathbb{S}(n,n)}(\tilde A(T(\frac{\epsilon}{3}),t_0,\tilde X_0)) - P_{\mathbb{S}(n,n)}(\tilde A(T(\frac{\epsilon}{3}),t_0,\bar X_0))\| \\
    & \leq \frac{\epsilon}{3}.
\end{align}
Now we use the triangle inequality at the time $k = T(\frac{\epsilon}{3})$:
\begin{equation}
\begin{aligned}
    & \|\mathbf{1}_{n,1} v(\bar X_0) - \mathbf{1}_{n,1} v(\tilde X_0)\|\\
    \leq~& \|\mathbf{1}_{n,1} v(\bar X_0) -  P_{\mathbb{S}(n,n)}(\tilde A(T(\frac{\epsilon}{3}),t_0,\bar X_0))\| \\
    &+ \|P_{\mathbb{S}(n,n)}(\tilde A(T(\frac{\epsilon}{3}),t_0,\bar X_0)) \\
    &- P_{\mathbb{S}(n,n)}(\tilde A(T(\frac{\epsilon}{3}),t_0,\tilde X_0))\| \\
     &+ \|\mathbf{1}_{n,1} v(\tilde X_0) -  P_{\mathbb{S}(n,n)}(\tilde A(T(\frac{\epsilon}{3}),t_0,\tilde X_0))\|\\
    \leq~& \frac{\epsilon}{3} + \frac{\epsilon}{3} + \frac{\epsilon}{3} = \epsilon.
\end{aligned}
\end{equation}
This concludes the proof. 
\hfill \qed

Having defined the continuous function $v(X_0)$, we provide a necessary and sufficient condition for asymptotic consensus of directions.  

\begin{theorem}\label{theorem:main}
Suppose the matrix $A$ together with $\bar \gamma$ satisfies Assumption~\ref{ass:100} and $X(t)$ is given by \eqref{eq:main_Algo_general}.
If and only if $v(X_0)X_0 \neq \bold{0}_{1,d}$, where $v(X_0)$ is the vector from Proposition~\ref{prop_new:3}, 
there is asymptotic consensus of directions. 

Furthermore, if $v(X_0)X_0 \neq \bold{0}_{1,d}$, there is also: 
\begin{enumerate}
    \item point-wise convergence of directions and the directions converge linearly to $\frac{v(X_0)X_0}{\|v(X_0)X_0\|_2}$ for all $i$;
    \item point-wise convergence of states. The sequence $\{X(t)\}$ converges (linearly) to  $P_{\mathbb{S}(n,d, \bar \gamma)}(\boldsymbol{1}_{n,1}(v(X_0)X_0))$ (if each $\gamma_i$ is differentiable at $\frac{v(X_0)X_0}{\|v(X_0)X_0\|_2}$). 
\end{enumerate}
\end{theorem}

\textbf{Proof:} Let us define $t=t_0+k$ and 
\begin{equation}
    z_i(t) = \frac{[\tilde{A}(k,t_0,X_0)]_i}{\|[\tilde{A} (k,t_0,X_0))]_i\|_2} \text{ and } y_i(t) = z_i(t)X_0,
\end{equation}
for all $i$ and $k > 0$. 
It may be readily concluded that $\frac{y_i(t)}{\|y_i(t)\|_2} = \frac{x_i(t)}{\|x_i(t)\|_2}$ for all $t > t_0$, and $\lim\limits_{t \rightarrow +\infty}y_i(t) = \lim\limits_{t \rightarrow +\infty}z_i(t)X_0 = y_{\infty} = v(X_0)X_0$. 

We first show sufficiency. Suppose $y_{\infty} = v(X_0)X_0 \neq \bold{0}_{1,d}$. Then $\lim\limits_{t \rightarrow +\infty}\frac{x_i(t)}{\|x_i(t)\|_2} = \lim\limits_{t \rightarrow +\infty}\frac{y_i(t)}{\|y_i(t)\|_2} = \frac{y_{\infty}}{\|y_{\infty}\|_2} =\frac{v(X_0)X_0}{\|v(X_0)X_0\|_2}$ for all $i$. Furthermore, continuity of the $\gamma_i$'s ensures  point-wise convergence of the states. If the sequence $\{z_i(t)\}$ converges linearly, so does the sequence $\{y_i(t)\}$. For large enough $t$ there exists $\delta \in (0,1)$ such that $\|y_i(t+1) - y_{\infty}\|_2 \leq \delta\|y_i(t) - y_{\infty}\|_2$.
The Taylor expansion of $g(y) = \frac{y}{\|y\|_2}$ at $y_{\infty}$ is
\begin{align}
    g(y_{\infty} + e) - g(y_{\infty}) = \frac{e}{\|y_{\infty}\|_2}\left(I_d - \frac{y_{\infty}^Ty_{\infty}}{\|y_{\infty}\|^2_2}\right ) + O(\|e\|_2^2).  
\end{align}
So, for large enough $t$, the norm $\|g(y(t)) - g(y_{\infty})\|_2$ is bounded by scaling of $\|y(t) - y_{\infty}\|_2$. 
An analogous argument can be made about linear convergence of states if each $\gamma_i$ is differentiable at $\frac{v(X_0)X_0}{\|v(X_0)X_0\|_2}$, since $x_i(t) = \gamma_i\left(\frac{y_i(t)}{\|y_i(t)\|_2}\right)\frac{y_i(t)}{\|y_i(t)\|_2}$ for all $i$. 

We now establish necessity. Suppose for an $X_0$, there is asymptotic consensus of directions so that \eqref{eq:5} is satisfied. But then there must be a time $T \geq t_0+n$ and a unit vector $h \in \mathbb{R}^{d}$ (as function of $X_0$) such that $ X(T)h^T > \bold{0}_{n,1}$. In other words, there must be a time $T$ so that all the directions are inside the interior of a half-sphere. Since $v(X(T))$ is positive, it holds that $v(X(T))X(T)h^T > 0$. Clearly $v(X(T))X(T) \neq \bold{0}_{1,d}$. We also know from previous results that $\tilde A(T-t_0, t_0, X_0)$ is a positive matrix for $T-t_0 \geq n$. Thus there must be $\bar j$ such that for any $\bar i$,
\begin{equation}
\begin{aligned}
0 < ~&\frac{|[v(X(T))X(T)]_{\bar j}|}{\|\tilde A(T-t_0, t_0, X_0)\|_{\text{F}}} \\
=~& \lim_{k \rightarrow \infty}   \frac{|[[\tilde A(k, T, X(T))]_{\bar i}\tilde A(T-t_0, t_0, X_0) X_0]_{\bar j}|}{\|[\tilde A(k, T, X(T))]_{\bar{i}}\|_2\|\tilde A(T-t_0, t_0, X_0)\|_{\text{F}}} \\
\leq~& \lim_{k \rightarrow \infty}   \frac{|[[\tilde A(k, T, X(T))\tilde A(T-t_0, t_0, X_0)]_{\bar i}X_0]_{\bar j}|}{\|[\tilde A(k, T, X(T))\tilde A(T-t_0, t_0, X_0)]_{\bar i}\|_2} \\
=~& |[v(X_0)X_0]_{\bar j}| \leq \|v(X_0)X_0\|_1, 
\end{aligned}
\end{equation}
where $\|\cdot \|_{\text{F}}$ denotes the Frobenius norm. 
This means that if there is asymptotic consensus of directions, then $v(X_0)X_0 \neq \bold{0}_{1,d}$.  $\hfill \qed$

In Theorem~\ref{theorem:main} it is stated that if asymptotic consensus of directions occurs, the convergence is linear and point-wise. So asymptotic consensus of directions without point-wise and linear order of convergence does not exist. 
Due to Theorem~\ref{theorem:main}, the following corollary is immediate. 
\begin{corollary}
Suppose the matrix $A$ together with $\bar \gamma$ satisfies Assumption~\ref{ass:100} and $X(t)$ is given by \eqref{eq:main_Algo_general}.
\begin{enumerate}
    \item The set of $X_0$'s for which there is asymptotic consensus of directions with point-wise convergence is an open set.
    \item The consensus direction (state) is locally continuous in $X_0$. 
    \end{enumerate}
\end{corollary}

We now provide a corollary with sufficient conditions for $v(X_0)X_0 \neq \bold{0}_{1,d}$ to be satisfied. 

\begin{corollary}\label{corollary:x}
Suppose the matrix $A$ together with $\bar \gamma$ satisfies Assumption~\ref{ass:100} and $X(t)$ is given by \eqref{eq:main_Algo_general}.
    Under each of the following conditions, $v(X_0)X_0 \neq \bold{0}_{1,d}$ is satisfied, whereby asymptotic consensus of directions with point-wise convergence is assured by Theorem~\ref{theorem:main}.

\begin{enumerate}
    \item The rows of $X_0\in \mathbb{S}(n,d, \bar \gamma)$ are in a half-space in the sense that there is a unit vector $h \in \mathbb{R}^d$ such that $\bold{0}_{n,1} \neq X_0h^T \geq \bold{0}_{n,1}$.
    \item The rank of $X_0 \in \mathbb{S}(n,d, \bar \gamma)$ is $n$, which is the case for all but a measure zero set of $X_0$'s if $d \geq n$. 
    \item If there is a nonzero column $c$ of $X_0\in \mathbb{S}(n,d, \bar \gamma)$ such that $c^T \not\in \mathcal{C}^{\perp}(A)$, where \begin{align}
     \mathcal{C}^{\perp}(A)\!=\!\{x \in  \mathbb{R}^n\!: \exists y \in  \mathcal{C}(A), \text{ s.t. } y \neq \bold{0}_{1,n},  yx^T = 0 \},   
    \end{align}
where $\mathcal{C}(A)$ is the minium cone for $A$. 
\end{enumerate}
\end{corollary}

\textbf{Proof:}\\
(1) Since $v(X_0)$ is a positive vector, $v(X_0)X_0h^T \neq 0$, which implies $v(X_0)X_0 \neq \bold{0}_{1,d}$. 

(2) If the rank of $X_0$ is $n$, obviously $v(X_0)X_0 \neq \bold{0}_{1,d}$. If $d \geq n$, the set of $X_0$'s with linearly dependent rows has measure zero. 

(3) If there is a column $c$ of $X_0$ such that $c^T$ is not in $\mathcal{C}^{\perp}(A)$, then since 
$v(X_0) \in \mathcal{C}(A)$, it holds that $v(X_0)c \neq 0$. 
$\hfill \qed$

Let us illustrate the necessary condition for consensus, i.e., $v(X_0)X_0 \neq \bold{0}_{1,d}$, by the following example. 
Let us consider the case when $X_0$ is such that there is no $h$ so that condition (1) in Corollary~\ref{corollary:x} is fulfilled, but there is an $h$ such that $\bold{0}_{n,1} = X_0h^T$. 

For such cases, we cannot guarantee $v(X_0)X_0 \neq \bold{0}_{1,d}$. 
In the plot on the left in Fig.~\ref{fig:bad_points_nd_22}, this is illustrated for the unit circle where $n = d = 2$.  Here $[X_0]_1 = -[X_0]_2$ meaning $[X_0]_1$ and $[X_0]_2$ have antipodal positions on the unit circle and are placed on a line that passes through the origin. 
Such a configuration comprises
a fixed point for Algorithm~\eqref{eq:main_Algo_general} (more conveniently be described by Algorithm~\eqref{eq:main_Algo_Lp} in this case).
One can verify that this fixed point is unstable.  In the right plot in Figure~\ref{fig:bad_points_nd_22}, the vector $v(X_0)$ is depicted by a blue arrow, whereas the two columns of $X_0$ are represented by black dots. Clearly $v(X_0)$ is orthogonal to the two columns, so asymptotic consensus of directions or states does not occur. 

\begin{figure}[h] 
    \centering
\includegraphics[width=0.9\columnwidth]{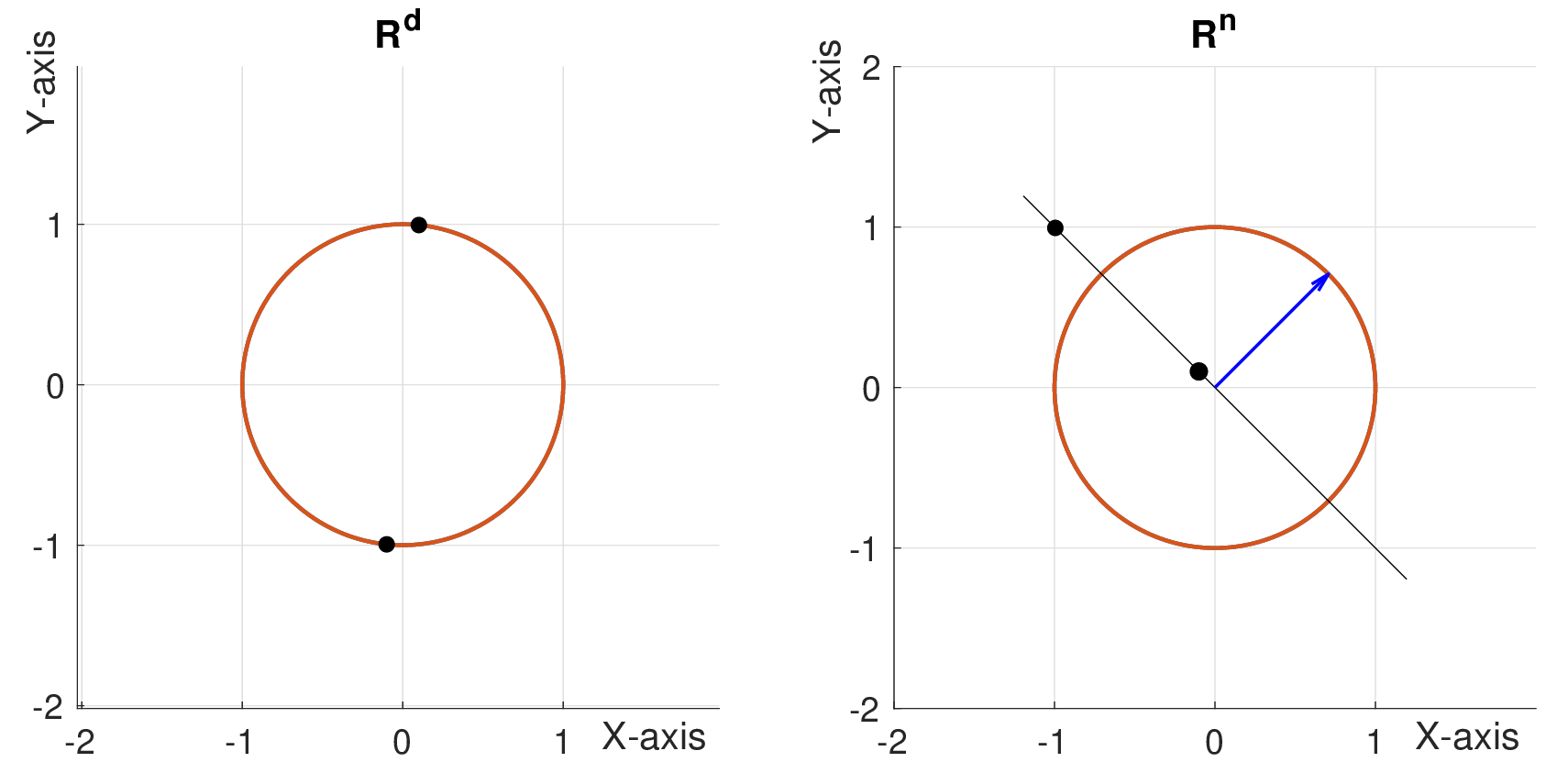} 
    \caption{Left: antipodal position of the rows of $X_0$ on the unit circle in $\mathbb{R}^d$. Right: the two columns of $X_0$ in $\mathbb{R}^n$ (black dots) and the vector $v(X_0)$ (blue arrow). Clearly $v(X_0)X_0 = \bold{0}_{1,d}$.}
    \label{fig:bad_points_nd_22}
\end{figure}

The plot on the right in Fig.~\ref{fig:bad_points_nd_22} illustrates that the condition $(1)$ in Corollary~\ref{corollary:x} can be equivalently be seen as a condition on the rows of $X_0$. 
If there exists an orthogonal matrix $R$ such that $X_0R$ has a column with all nonnegative entries and at least one entry that is strictly positive, then $v(X_0)X_0 \neq \bold{0}_{1,d}$. 

Condition $(3)$ in Corollary~\ref{corollary:x} is illustrated by Fig.\ref{fig:shrinking_cones_2d}, where the evolution of the $\mathcal{C}(P_{\mathbb{S}(n,n)}(\tilde A(k,t_0,X_0)))$-cones is shown. The initial cone $\mathcal{C}(A)$ corresponds to bright red and this cone contains all subsequent $\mathcal{C}(P_{\mathbb{S}(n,n)}(\tilde A(k,t_0,X_0)))$-cones. As $k$ increases, the darkness of the red color increases. Note that while the cones extend infinitely, the visualization is limited to the portion of the cone that intersects the unit disc. 
The set $\mathcal{C}^{\perp}(A)$ is illustrated by the cropped  purple region. 
Since the cones formed by the rows of $P_{\mathbb{S}(n,n)}(\tilde A(k,t_0,X_0))$ shrink with time, the corresponding geodesic hull also gradually shrinks to a single point. Fig. \ref{fig:shrinking_cones_3d} on the right shows an example of how the geodesic hull shrinks with time on $\mathbb{S}^2$. On the left, the convergence of the points representing the normalized rows $P_{\mathbb{S}(n,n)}(\tilde A(k,t_0,X_0))$ on the unit sphere is shown. As time progresses, these points move closer together, ultimately converging to a single point, which corresponds to the vector $v(X_0)$.

\begin{figure}[h]
    \centering
\includegraphics[width=0.65\columnwidth]{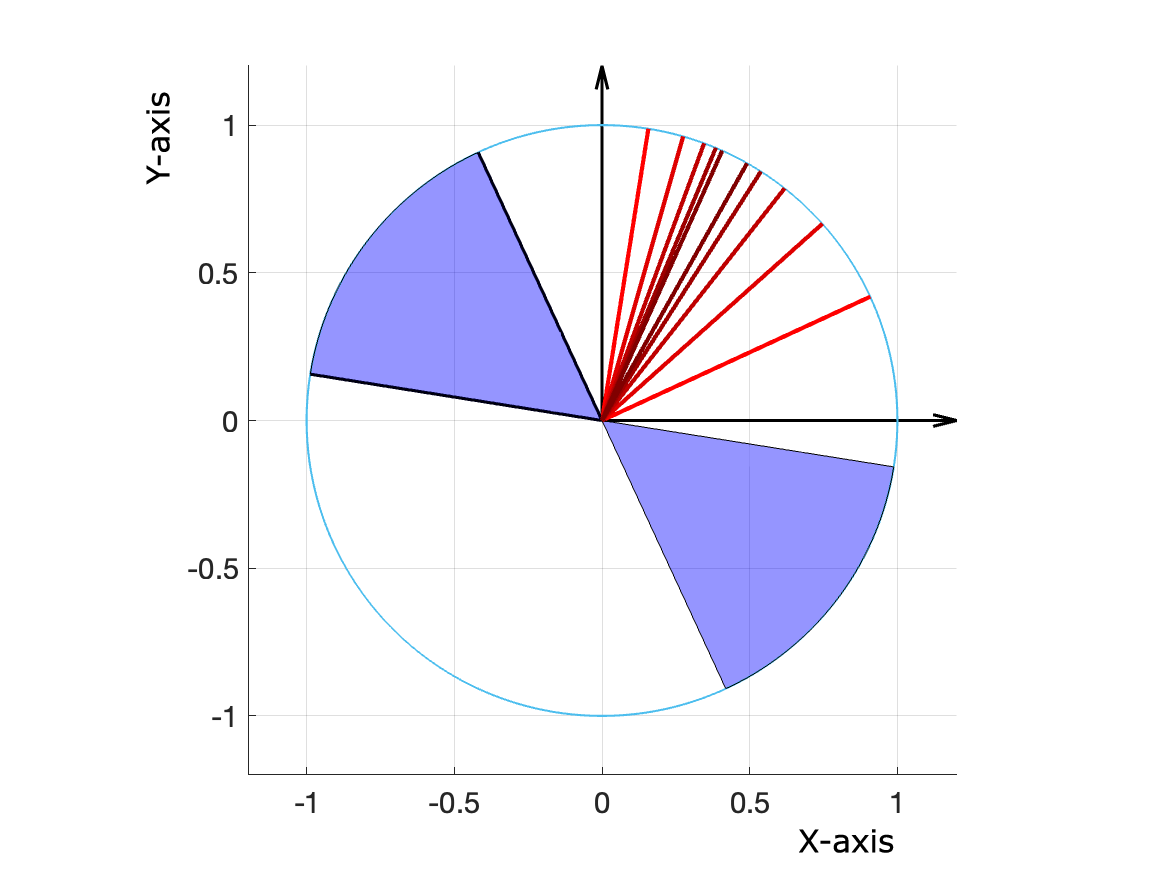} 
    \caption{The evolution of the minimum cone $\mathcal{C}(\tilde A(k,t_0,X_0))$ in $\mathbb{R}^2$. The first $5$ time steps are presented, with darker colors corresponding to later time steps. The purple region represents  $\mathcal{C}^{\perp}(A)$. Here, $d=n=2$.}
    \label{fig:shrinking_cones_2d}
\end{figure}

\begin{figure}[h] 
    \centering
\includegraphics[width=0.95\columnwidth]{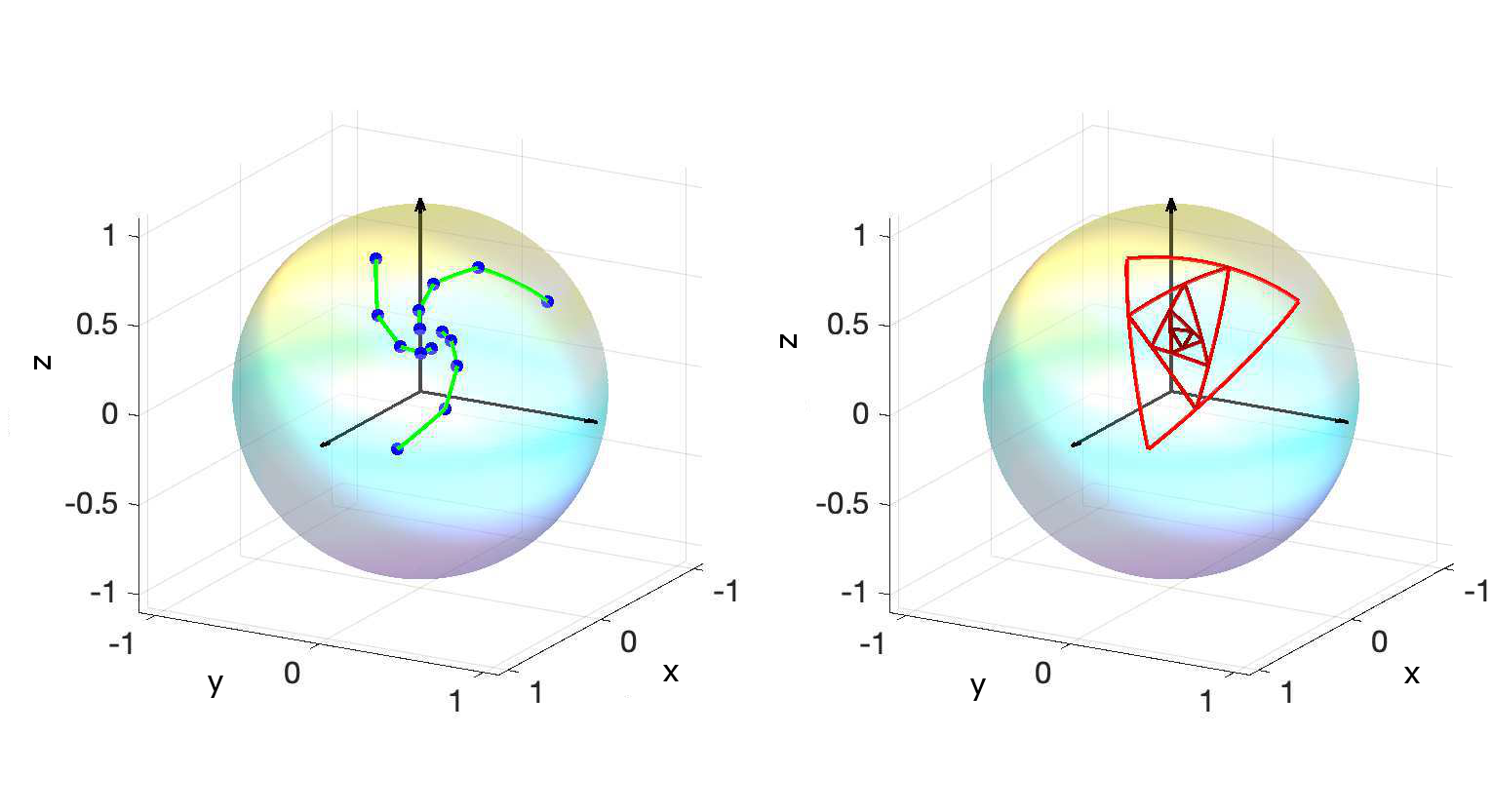} 
    \caption{Left: The convergence of the points representing the normalized rows $P_{\mathbb{S}(n,n)}(\tilde A(k,t_0,X_0))$ on the unit sphere in $\mathbb{R}^3$. Right: The corresponding geodesic hulls are shrinking with time. The first $5$ time steps are presented, with darker colors corresponding to later time steps. Here, $d=n=3$.}
    \label{fig:shrinking_cones_3d}
\end{figure}

\section{Simulations}\label{sec:simulations}
This section presents various examples of asymptotic consensus of directions or states for Algorithm \eqref{eq:main_Algo_general}. We conducted numerous simulations that numerically support the results presented in Section \ref{sec:convergence_to_consensus}. 
Corollary~\ref{corollary:x} provides sufficient conditions for the $v(X_0)X_0 \neq \bold{0}_{1,d}$ in Theorem~\ref{theorem:main}. However, numerical simulations for random system parameters seem to suggest a conjecture that the condition $v(X_0)X_0 \neq \bold{0}_{1,d}$ holds for all $X_0$ but a set of measure zero. This conclusion is based on $10^8$ simulations carried out across various randomly selected configurations, including different space dimensions $d$, numbers of agents $n$, matrix structures \(A\), and non-trivial star boundaries including $l_p$-spheres, where asymptotic consensus of directions was observed in every experiment. 

In each simulation, $n$ and $d$ were selected randomly. For each simulation a strongly connected graph $\mathcal{G}(n)$ was selected randomly together with an $A$-matrix satisfying Assumption~\ref{ass:100}. To construct non-trivial star-convex shapes, $K$ random unit vectors in $\mathbb{R}^d$ were generated, with the number of vectors $K$ chosen uniformly from $d$ to $4d$. Each vector was then scaled by a random positive factor. The distance to the boundary in any given direction $h \in \mathbb{R}^d$ was then computed as the sum of some minimum distance and the nonnegative scalar products between the direction vector $h$ and each of the $K$ unit vectors, raised to a varying integer power sampled uniformly from $1$ to $10$. The initial $X_0$'s were generated by element-wise sampling from the standard normal distribution followed by projection onto the star boundaries.

Some examples of the simulations are presented by Figures \ref{fig:lp_spheres_r2}-\ref{fig:star_boundary}. For visualization purposes, the space dimension $d$ was chosen to be $2$ or $3$. Figures \ref{fig:lp_spheres_r2},\ref{fig:lp_spheres_r3} illustrate various $l_p$-spheres with different values of the norm parameter $p$ and radius $r$. In Fig.~\ref{fig:linear_rate}, linear convergence is shown for the configuration in Fig. \ref{fig:lp_spheres_r3}, left.
An example of convergence to consensus on a more complex star boundary is presented in Fig.\ref{fig:star_boundary}. 
In Fig.\ref{fig:lp_spheres_r2} and Fig.\ref{fig:star_boundary}, the initial positions of the agents, $x_i$-states, are presented on the left while the final positions are shown on the right, along with consensus direction. Fig.\ref{fig:lp_spheres_r3} illustrates both the initial and final positions on a single figure, along with the trajectories of each agent from their initial to final positions. In all the presented examples, $n>d$, and for all cases except Fig.\ref{fig:lp_spheres_r3}, left, the initial matrix $X_0$ does not belong to a half-space. 

\begin{figure}[h]
    \centering
\includegraphics[width=0.9\columnwidth]{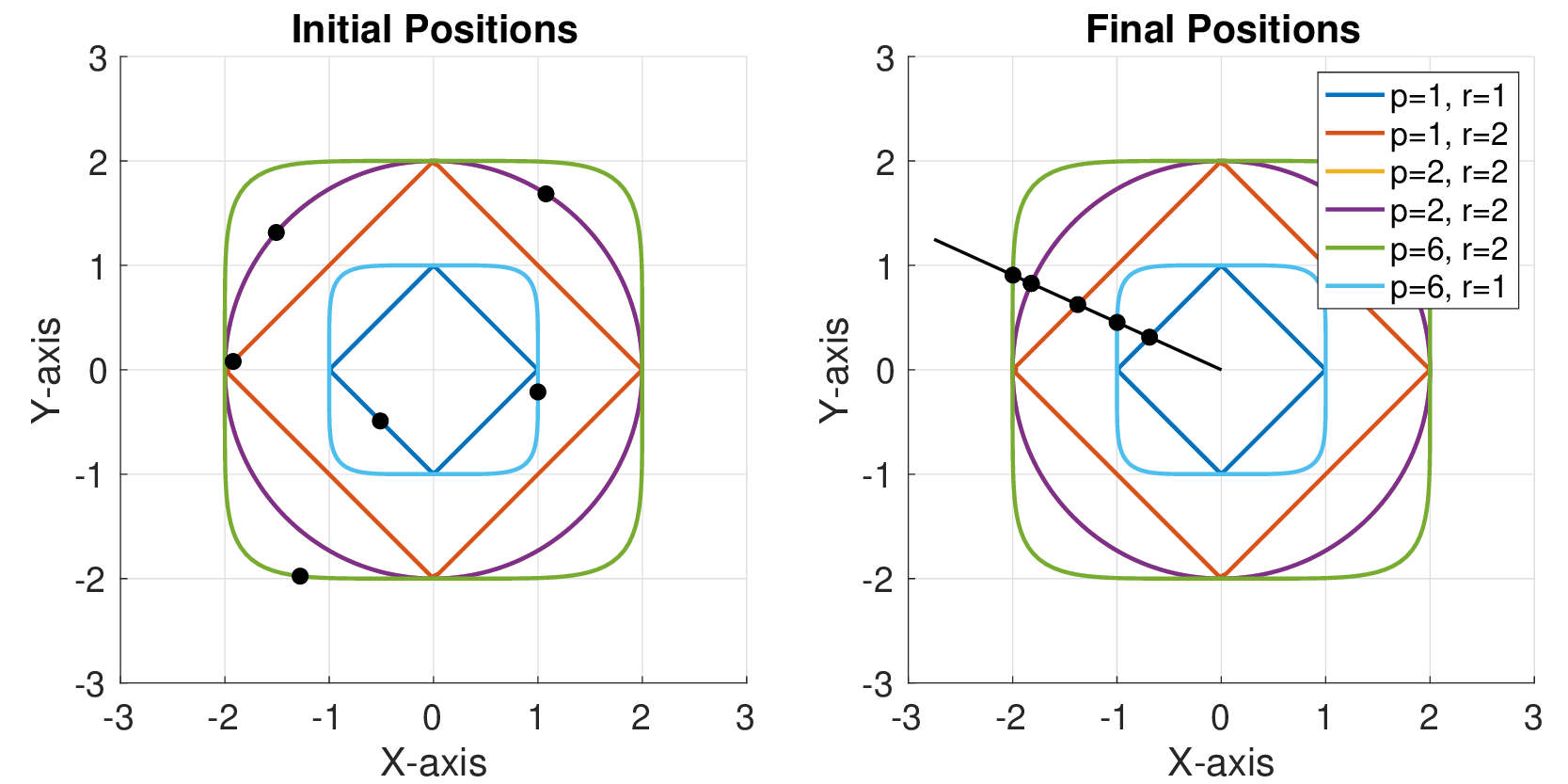}
    \caption{Convergence to the consensus direction of agents moving on various $l_p$-spheres in $\mathbb{R}^2$. The chosen radius and parameter $p$ for the spheres are indicated in the legend. Left: initial positions of the agents. Right: final positions of the agents, which have converged to the consensus direction. The line segment connecting the origin and final positions of the agents illustrates the consensus direction.}
     \label{fig:lp_spheres_r2}
\end{figure}

\begin{figure}[h]
    \centering
\includegraphics[width=0.9\columnwidth]{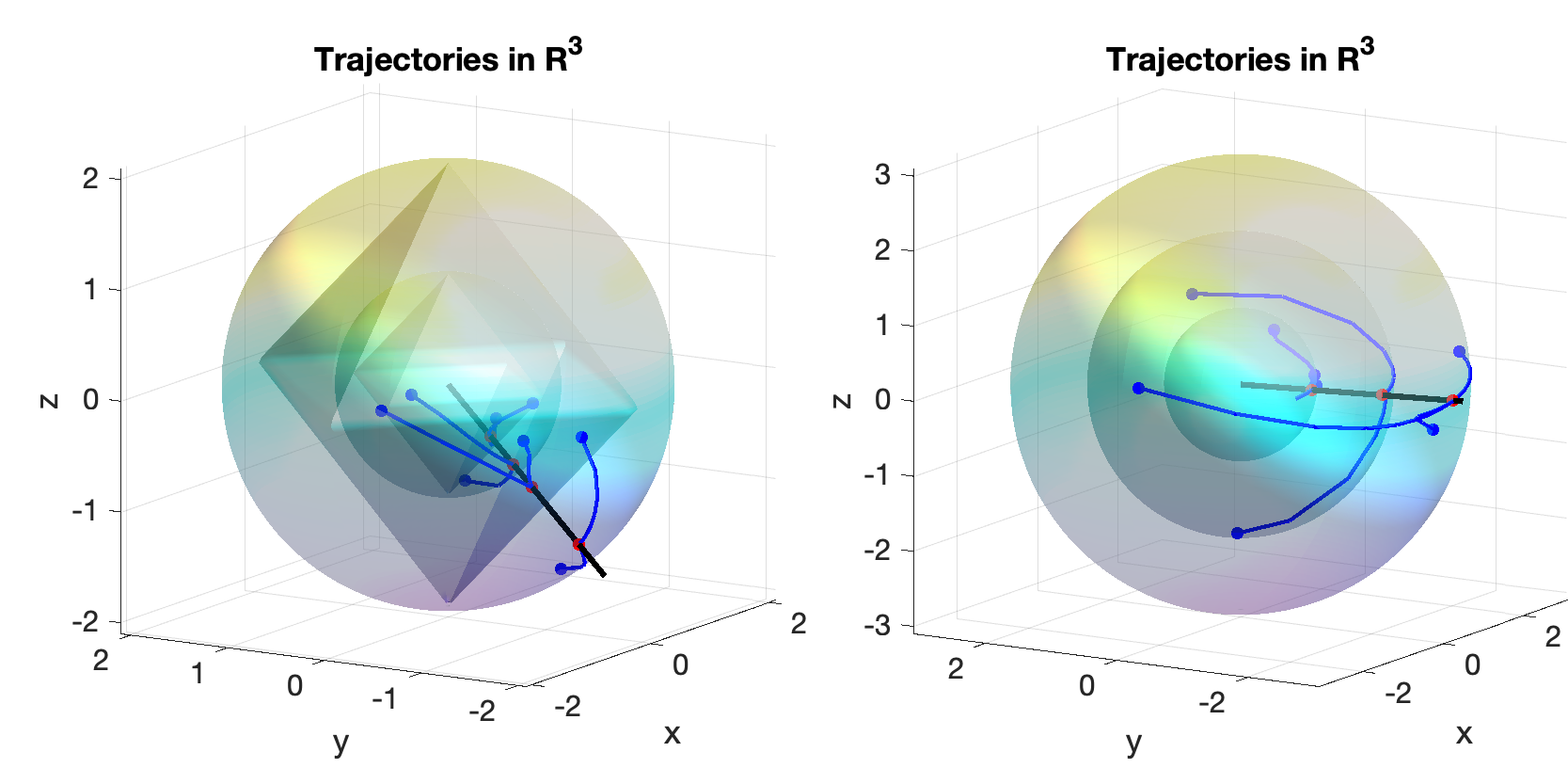}   
\caption{Convergence to the consensus direction of agents moving on various $l_p$-spheres in $\mathbb{R}^3$, where on the left: $p=\{1,2\}$, $r=\{1,2\}$; right: $p=2$, $r=\{1,2,3\}$. 
Blue dots represent the initial positions of the agents, whereas red dots -- the final positions. The blue lines trace the trajectories of each agent over time. The line segment connecting the origin and the final positions of the agents illustrates the consensus direction.}
    \label{fig:lp_spheres_r3}
\end{figure}

\begin{figure}[h] 
    \centering
\includegraphics[width=0.75\columnwidth]{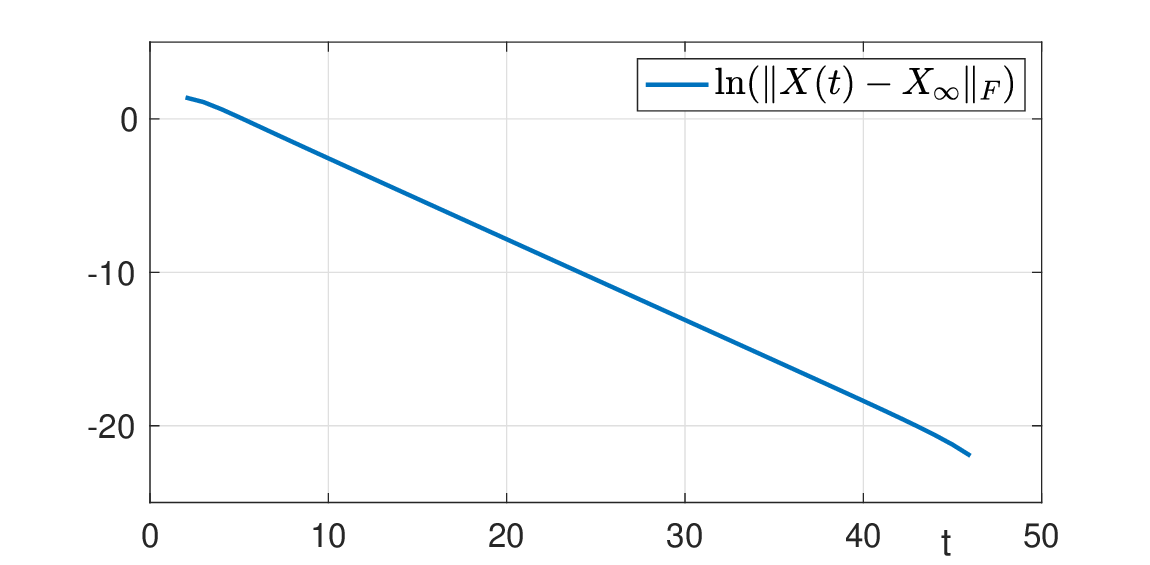} 
    \caption{Linear convergence is shown for the configuration in Fig. \ref{fig:lp_spheres_r3}, left. 
    }
    \label{fig:linear_rate}
\end{figure}

\begin{figure}[h]
    \centering
\includegraphics[width=0.9\columnwidth]{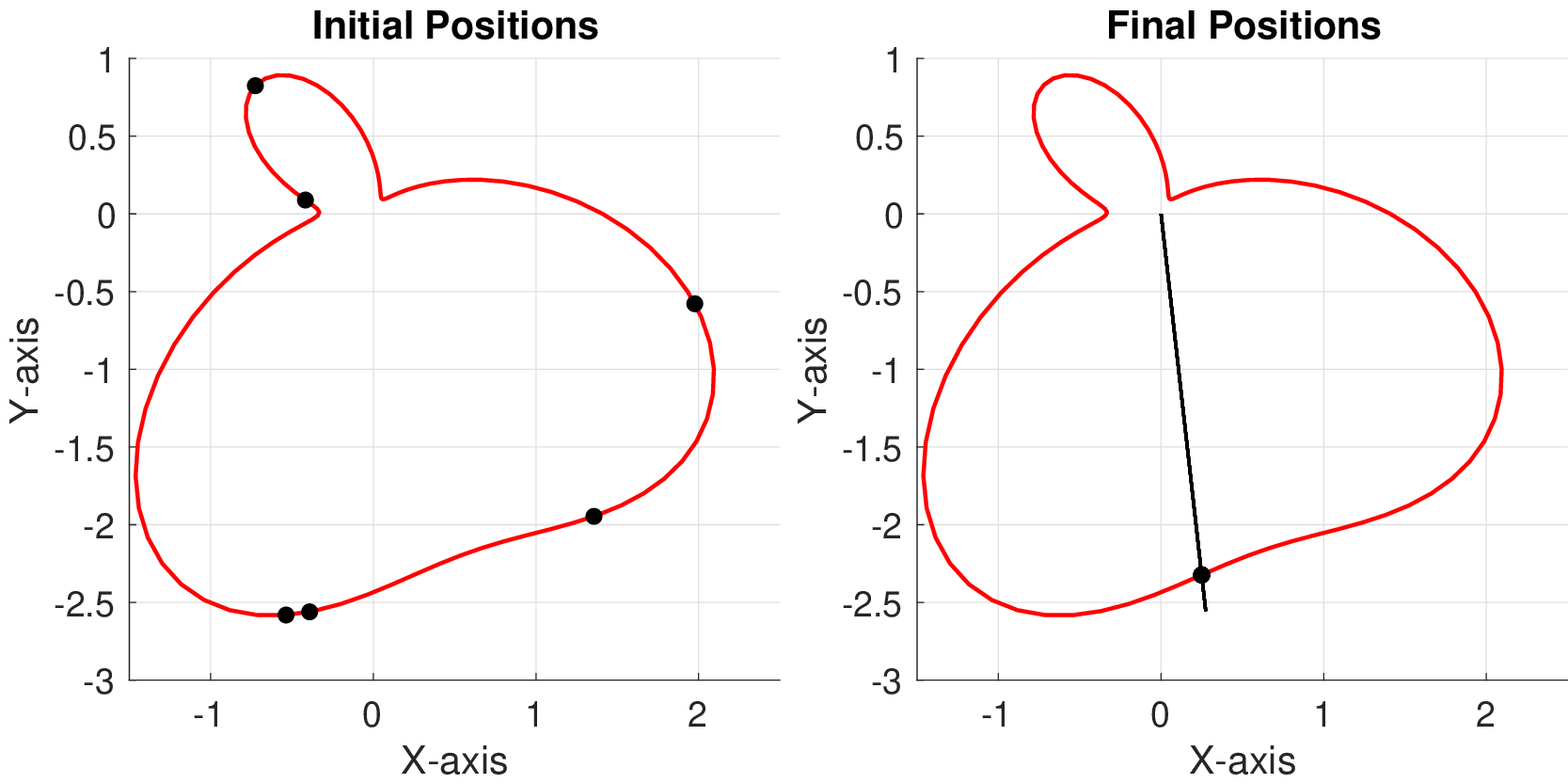} 
    \caption{Convergence to consensus on a star boundary $\mathbb{S}(6,2, \bar \gamma)$ in $\mathbb{R}^2$. Left: initial positions of agents. Right: final position of agents, which have converged to the same point within $\mathbb{S}(6,2, \bar \gamma)$. The depicted line segment connects the origin and the consensus point.}
    \label{fig:star_boundary}
\end{figure}

\section {Conclusions} \label{sec:conclusions}
The focus of this paper is on discrete-time consensus algorithms for multi-agent systems where the state of each agent resides on the boundary of a star-convex set (star boundary). A special case occurs, for example, when all the agent-specific sets are the same and equal to the unit Euclidean sphere. 
The state update map at each discrete time-step comprises a radial projection of a conical or convex combination of neighbors' states onto the agent specific star boundary. The neighborhoods are defined by a directed strongly connected graph. 
A necessary and sufficient condition for asymptotic consensus of directions (normalized states) is provided. Furthermore, if asymptotic consensus of directions occurs, the directions converge linearly. A collection of sufficient conditions under which this condition is satisfied is also presented.

\begin{ack}                              
This work was partially supported by the Wallenberg AI, Autonomous Systems and Software
Program (WASP) funded by the Knut and Alice Wallenberg Foundation and partly supported by the Swedish Research
Council under Grant 2019-04769. 
\end{ack}

\bibliographystyle{elsarticle-harv}      
\bibliography{refs}            

\end{document}